\documentclass[12pt]{article}
\usepackage[letterpaper,top=2cm,bottom=2cm,left=3cm,right=3cm,marginparwidth=1.75cm]{geometry}

\usepackage{graphicx} 
\usepackage{subfigure}
\usepackage[colorlinks=true, allcolors=blue]{hyperref}
\usepackage{amssymb}
\usepackage{amsmath}
\usepackage{amsthm}
\usepackage{mathdots}
\usepackage{xcolor}
\usepackage[T1]{fontenc}

\def\1{\mathbf 1}
\def\R{\mathbb{R}}
\newtheorem{theorem}{Theorem}
\newtheorem{proposition}{Proposition}
\newtheorem{corollary}{Corollary}
\newtheorem{lemma}{Lemma}

\title{The Laplacian matrix of weighted threshold graphs}

\author{Yingyue Ke\thanks{Faculty of Electrical Engineering,
Mathematics and Computer Science, P.O Box 5031, 2600 GA Delft, The Netherlands; \emph{email}: y.y.ke@tudelft.nl}, ~Willem H. Haemers\thanks{Tilburg University, Tilburg, The Netherlands; \emph{email}: haemers@tilburguniversity.edu} 
\ and Piet Van Mieghem\thanks{Faculty of Electrical Engineering,
Mathematics and Computer Science, P.O Box 5031, 2600 GA Delft, The Netherlands; \emph{email}: P.F.A.VanMieghem@tudelft.nl} }

\begin{document}
\maketitle

\begin{abstract}

Threshold graphs are generated from one node by repeatedly adding a node that links to all existing nodes or adding a node without links.   
In the weighted threshold graph, we add a new node in step $i$, which is linked to all existing nodes by a link of weight $w_i$.
In this work, we consider the set ${\cal A}_N$ that contains all Laplacian matrices of weighted threshold graphs of order $N$. 
We show that ${\cal A}_N$ forms a commutative algebra.
Using this, we find a common basis of eigenvectors for the matrices in ${\cal A}_N$.
It follows that the eigenvalues of each matrix in ${\cal A}_N$ can be represented as a linear transformation of the link weights.
In addition, we prove that, if there are just three or fewer different weights, two weighted threshold graphs with the same Laplacian spectrum must be isomorphic.

\end{abstract}

\section{Introduction}\label{sec1}

The adjacency matrix $A$ of an unweighted graph $G$ with nodes $\{1,\ldots,N\} $ is a matrix matrix $N \times N$ with elements $a_{ij}$, where $a_{ij}=1$ if there is a link between node $i$ and node $j$, otherwise $a_{ij}=0$. 
We use notation and notions from  \cite{van2023graph}.
For a weighted graph each link $\{i,j\}$ has a weight $a_{ij}\in\R$.
We indentify a non-link with a weight zero link.
Then the elements of the (weighted) adjacency matrix $A$ are the link weights $a_{ij}$.
The Laplacian matrix of a (weighted) graph is defined by $Q=\Delta-A$, where $\Delta = {\rm{diag}}(d_1, d_2,...,d_N )$ and $d_i = \sum_{j=1}^N a_{ij}$ is the (weighted) degree of node $i$.  
The Laplacian matrix necessarily has a zero eigenvalue because its row sum is zero.

A threshold graph is a graph obtained from one node by repeatedly adding an isolated node or a dominant node, where an isolated node is not connected to any other node in a graph and a dominating node is a node that links to all other nodes in a graph. 
The concept of threshold graphs is introduced independently in \cite{hammer1977aggregation} and \cite{henderson1977graph}.
A comprehensive review of threshold graphs is provided in \cite{mahadev1995threshold}.
Formulas for the Laplacian spectrum and the number of spanning trees in a threshold graph is given in \cite{hammer1996laplacian}.
The application of threshold graphs in building real-world networks is discussed in \cite{bustos2012dynamics, konig2014nestedness, liu2014control}.

In the weighted threshold graph, we start with one node and add a new node in step $i$ ($i=2,\ldots,N$), which is linked to all existing nodes by a link of weight $w_i\in\R$.
We label the nodes as $1,2,3,\ldots,N$ according to the order in which the nodes are added. 
Then the $N \times N$ adjacency matrix $A$ of a weighted threshold graph equals
\begin{equation}
        A=\begin{bmatrix}
            0 & w_{2} & w_{3} & \cdots & \cdots & w_{N} \\
            w_{2} & 0 & w_{3} & \cdots & \cdots & w_{N} \\
            w_{3} & w_{3} & 0 & \cdots & \cdots & w_{N} \\
            \vdots & \vdots & \vdots & \ddots &  & \vdots \\
            \vdots & \vdots & \vdots &  & \ddots &  w_{N}\\
            w_{N} & w_{N} & w_{N} & \cdots & w_{N} & 0\\
    \end{bmatrix}
    \label{eq:wa}
\end{equation}
Thus the vector $W=(w_2,w_3, \ldots, w_N)$ 
determines the weighted threshold graph.
We call $W$ the weight vector and write $G_W$ for the corresponding weighted threshold graph.

The row sums of the adjacency matrix $A$ in Eq.~(\ref{eq:wa}) give the degrees of the nodes in $G_W$, i.e., 
\begin{equation}\label{eq:wd}
    d_i = (i-1)w_i + \sum^N_{j=i+1}w_j, ~\text{for}~1\leq i \leq N .
\end{equation}
The Laplacian matrix $\Delta-A$ of $G_W$ is written as 
\begin{equation}
        Q_W=\begin{bmatrix}
            d_1 & -w_{2} & -w_{3} & \cdots & \cdots & -w_{N} \\
            -w_{2} & d_2 & -w_{3} & \cdots & \cdots & -w_{N} \\
            -w_{3} & -w_{3} & d_3 & \cdots & \cdots & -w_{N} \\
            \vdots & \vdots & \vdots & \ddots &  & \vdots \\
            \vdots & \vdots & \vdots &  & \ddots &  -w_{N}\\
            -w_{N} & -w_{N} & -w_{N} & \cdots & -w_{N} & d_N\\
    \end{bmatrix}
    \label{eq:wl}
\end{equation}

Figure \ref{fig} illustrates the threshold graph $G_W$ with $N=6$ nodes coded by weight vector $W=(1,0,-\sqrt{2},0,~2)$ and its Laplacian matrix $Q_W$.
The graph $G_W$ is constructed by sequentially adding node $i$, along with its associated weight $w_i$ for $1 \leq i \leq 6$ (see Figure \ref{graph construction}).

\begin{figure}[h]
    \centering
    \includegraphics[width=\linewidth]{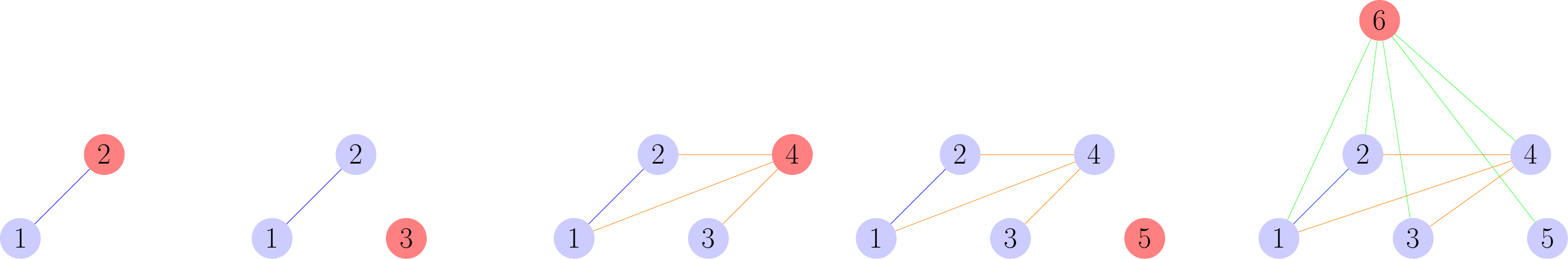}
    \caption{Construction of $G_W$ coded by weight vector $W=(1,0,-\sqrt{2},0,~2)$. At each step, a node $i$ (highlighted in red) is added with links of weight $w_i$, which connect it to all previous nodes $j <i $. Line colors represent link weights. Blue, orange and green lines have weights 1, $-\sqrt{2}$ and 2, respectively.}
    \label{graph construction}
\end{figure}

\begin{figure}[h]
    \subfigure{
    \includegraphics[height=100pt]{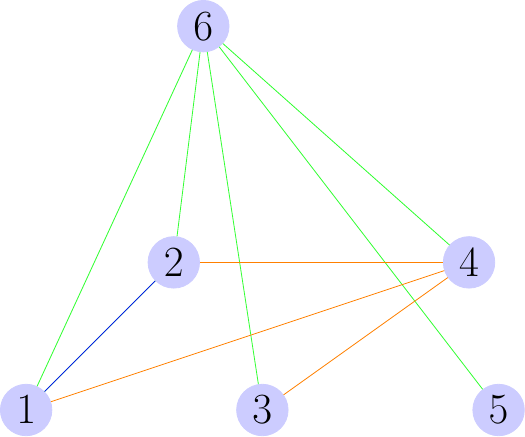} }
    ~\\[-105pt]
    \subfigure{~\hspace{145pt}
    $   Q_W = \left[
        \begin{array}{cccccc}
        3-\sqrt{2} & -1 & 0 & \sqrt{2} & 0 & -2 \\
        -1 & 3-\sqrt{2} & 0 & \sqrt{2} & 0 & -2 \\
        0 & 0 & 2-\sqrt{2} & \sqrt{2} & 0 & -2 \\
        \sqrt{2} & \sqrt{2} & \sqrt{2} & 2-3\sqrt{2} & 0 & -2 \\
        0 & 0 & 0 & 0 & 2 & -2 \\
        -2 & -2 & -2 & -2 & -2 & ~10
        \end{array}\right]
   $}\caption{The threshold graph $G_W$ with weight vector $W=(1,0,-\sqrt{2},0,~2)$ and its Laplacian matrix $Q_W$.}
    \label{fig}
\end{figure}

We define ${\cal A}_N$ to be the set of Laplacian matrices of weighted threshold graphs of order $N$. 
The main purpose of this paper is to investigate the structure of ${\cal A}_N$ and the eigensystem of the matrices in ${\cal A}_N$.
In Section~\ref{sec:an}, we establish that ${\cal A}_N$  constitutes a commutative algebra of dimension $N-1$ and provide a basis. 
In Section~\ref{sec:eigen}, we find eigenvectors and eigenvalues for the matrices in ${\cal A}_N$.
In particular we prove that the spectrum of $Q_W\in{\cal A}_N$ can be written as a linear transformation of its weight vector $W$.
In section~\ref{sec:cos}, we show that, in case of at most three values for the weights, two weighted threshold graphs are isomorphic if their Laplacian matrices are cospectral, i.e., they have the same eigenvalues.

\section{Structure of ${\cal A}_N$} \label{sec:an}
\subsection{A basis}
Considering the threshold graph with $N$ nodes and weight vector $W_i=(0,\ldots,0,w_i=1,0,\ldots,0)$, where $2\leq i\leq N$, the corresponding Laplacian matrix $Q_i:=Q_{W_i}$ is given by
    \begin{equation}
    \label{eq:qi}
        \left( Q_{i} \right)_{mn}= 
        \begin{cases}
            -1 &1 \leq n < m=i,\text{ or } 1 \leq m< n=i, \\
            1 & 1 \leq m = n < i,\\
            i-1 & m=n=i, \\
            0  & \text{otherwise}.
        \end{cases}
    \end{equation}  
    
For any threshold graph $G_W$ of $N$ nodes coded by $W=(w_2,w_3, \ldots, w_N)$, its Laplacian matrix $Q_W$ can be represented by 
    \begin{equation}\label{eq:qt}
        Q_W = \sum_{i=2}^{N} w_iQ_i,
    \end{equation}  
which shows that ${\cal A}_N=\text{ span}\{Q_2,\ldots ,Q_N\}$.
Clearly, the matrices $Q_2,\ldots,Q_N$ are independent, 
which implies that $\{Q_2,\ldots,Q_N\}$ is a basis for ${\cal A}_N$ which is an $(N-1)$-dimensional linear subspace of ${\R}^{N\times N}$, the vector space of real $N\times N$ matrices.

\subsection{The algebra ${\cal A}_N$}
In this subsection, we establish that ${\cal A}_N$ is an $(N-1)$-dimensional commutative subalgebra of $\R^{N\times N}$.
We start with some properties of the matrices $Q_i$ defined in Eq.~(\ref{eq:qi}).
\begin{lemma}\label{lm:qic}
    \begin{equation*}\label{eq:qij}
        Q_iQ_j=Q_jQ_i = Q_i.
        \text{~~for }
        ~2\leq i < j\leq N.
    \end{equation*}
\end{lemma}
\begin{proof}
When $i<j$, every row (or column) of $Q_i$ is orthogonal with every column (or row) of $Q_j-I_N$.
Therefore $Q_i (Q_j - I_N) = (Q_j - I_N)Q_i = O$, and hence $Q_iQ_j=Q_jQ_i = Q_i$.
\end{proof}

\begin{lemma}\label{lm:qip}
\begin{equation}\label{eq:qip}       
         Q_i^p=i^{p-1}Q_i-\frac{i^{p-1}-1}{i-1}\sum_{j=2}^{i-1}Q_j 
         \text{~~for }
         1 \leq p,~2\leq i\leq N.
\end{equation}
\end{lemma}

\begin{proof}
If $p=1$, then Eq. (\ref{eq:qip}) is obvious.
If $p=2$, we have
\begin{equation*}
        \left( Q_i^2 \right)_{mn}= \sum_{k=1}^{N} \left( Q_i \right)_{mk}\left( Q_i \right)_{kn}=
        \begin{cases}
            1 
            &\text{$1 \leq m \neq n < i$,}\\
            2 
            &\text{$1 \leq m=n <i$,}\\
            i(i-1) 
            &\text{$m=n=i$,}\\
            -i 
            &  \text{$1 \leq n < m=i$ or $ 1 \leq m< n =i$,}\\
            0  & \text{otherwise.}
        \end{cases}
\end{equation*}

For $(iQ_i-\sum_{j=2}^{i-1}Q_j)_{mn}$ we find the same values.
Therefore, $Q_i^2 = iQ_i-\sum_{j=2}^{i-1}Q_j$.\\
Next, suppose that Eq. (\ref{eq:qip}) holds. Using Lemma~~\ref{lm:qic}, we obtain 
\begin{equation*}
    \begin{aligned}
            Q_i^{p+1}
            &= Q_i^{p}Q_i=\left( i^{p-1}Q_i-\frac{i^{p-1}-1}{i-1}\sum_{j=2}^{i-1}Q_j \right)Q_i 
            =i^{p-1}Q_i^2-\frac{i^{p-1}-1}{i-1}\sum_{j=2}^{i-1}Q_j \\
            &=i^{p-1} \left( iQ_i-\sum_{j=2}^{i-1}Q_j\right)-\frac{i^{p-1}-1}{i-1}\sum_{j=2}^{i-1}Q_j 
            =i^{p} Q_i -\frac{i^{p}-1}{i-1}\sum_{j=2}^{i-1}Q_j.
    \end{aligned}
\end{equation*}
Thus, Eq.~(\ref{eq:qip}) holds for $p\geq 3$ by induction.
\end{proof}

Recall that $Q_W\in{\cal A}_N$ with weight vector $W=(w_2,\ldots,w_N)$ can be written as
\begin{equation}\label{eq:span}
Q_W=\sum_{i=2}^N w_iQ_i.
\end{equation}

Lemma~\ref{lm:qic} shows that any two matrices $Q_i$ and $Q_j$ commute, i.e., $Q_iQ_j=Q_jQ_i$.
Moreover, Lemma~\ref{lm:qic} and \ref{lm:qip} show that $Q_iQ_j\in{\cal A}_N$ for $2\leq i,j\leq N$.
Using Eq.~(\ref{eq:span}) it follows that any two matrices in ${\cal A}_N$ commute, and that ${\cal A}_N$ is closed under multiplication.
Thus, we have established that 

\begin{theorem}\label{th3}
${\cal A}_N$ is an $(N-1)$-dimensional commutative subalgebra of $\R^{N\times N}$.
\end{theorem}

Take $Q_W, Q_{W'}\in{\cal A}_N$ 
with weight vectors $W=(w_2,\ldots,w_N)$ and $W'=(w'_2,\ldots,w'_N)$.
Then $Q_W+Q_{W'}\in{\cal A}_N$ with weight vector $W+W'$, 
and similarly $aQ_W\in{\cal A}_N$ with weight vector $aW$ for every $a\in\R$.
Also, the product $Q_W Q_{W'}\in {\cal A}_N$, but more work is required to obtain the weight vector of $Q_W Q_{W'}$ from $W$ and $W'$.
\begin{proposition}\label{prp1}
$\displaystyle{
Q_W Q_{W'} =
\sum_{i=2}^{N} \left( i w_i  w_i^{\prime} 
            - \sum_{j=i+1}^{N} \left(  w_j  w_j^{\prime} -w_i w_j^{\prime} - w_j  w_i^{\prime} \right) \right)Q_i}$.
\end{proposition}
\begin{proof}
Using Lemma~\ref{lm:qic}, we start from Eq.~(\ref{eq:span}),
\begin{equation}\label{eq:qw2}
    Q_W Q_{W'}= \sum_{i=2}^{N} w_i  w_i^{\prime} Q_i^2 
            +\sum_{i=2}^{N-1}\sum_{j=i+1}^{N} \left( w_i  w_j^{\prime} +w_j  w_i^{\prime}\right) Q_i.
\end{equation}
We invoke Lemma~\ref{lm:qip} and obtain
\begin{equation}\label{eq:w2q2}
    \sum_{i=2}^{N} w_i  w_i^{\prime} Q_i^2 
    = \sum_{i=2}^{N} i w_i  w_i^{\prime} Q_i 
    -\sum_{i=2}^{N} \sum_{j=2}^{i-1}\ w_i  w_i^{\prime} Q_j.
\end{equation}
Reversing the $i$ and $j$ sums in~(\ref{eq:w2q2}) yields
\begin{equation} \label{eq:w2q}
    \sum_{i=2}^{N} \sum_{j=2}^{i-1}\ w_i  w_i^{\prime} Q_j=\sum_{j=2}^{N-1} \sum_{i=j+1}^{N} w_i  w_i^{\prime} Q_j=\sum_{i=2}^{N-1} \sum_{j=i+1}^{N} w_j  w_j^{\prime} Q_i.
\end{equation}
Substituting (\ref{eq:w2q}) and (\ref{eq:w2q2}) into (\ref{eq:qw2}) proves Proposition~\ref{prp1}.
\end{proof} 

\section{An eigensystem}

\label{sec:eigen}
It is known (see for example \cite{gantmacher}) that a set of mutually commuting symmetric matrices has a common basis of eigenvectors.
In this section, we will find these common eigenvectors for the matrices in ${\cal A}_N$. 
As a result, the eigenvalues of each $Q_W \in {\cal A}_N$ can be obtained by a linear transformation of its weight vector $W$.
We have the following lemma on the eigenvalues and eigenvectors for each matrix $Q_i$ defined in Eq.~(\ref{eq:qi}).
\begin{lemma}\label{lm:1}
Let $v_1$ be the all-one vector $\1$ and for $2 \leq j \leq N$, we 
    define the $N \times 1$ vector $v_j$ as
     \begin{equation}
            \left( v_{j} \right)_{k}= 
            \begin{cases}
                1 &  \text{$k < j$,}\\
                1-j &  \text{$ k=j$, }\\
                0 & \text{$k > j$.}
            \end{cases}
     \end{equation}
     Then,
      $\left\{ v_1, \cdots, v_N\right\}$ 
is an orthogonal basis of eigenvectors for each $Q_i$. 
The corresponding eigenvalues $Q_{i}v_{j} = \left(  \mu_i\right)_j v_{j}$ are
        \begin{equation}
           \left(  \mu_i\right)_j =
            \begin{cases}
                1 & \text{$2 \leq j <i$,}\\
                i & \text{$2 \leq j=i$,}\\
                0 & \text{$j>i$ or $j=1$.}\\
            \end{cases}
        \end{equation}
    and the characteristic polynomial is $P_{Q_i} \left( \lambda \right) =  \left( \lambda -1  \right)^{i-2} \left( \lambda -i  \right) \lambda^{N-i+1}$.
\end{lemma}

\begin{proof}
The vector $Q_i v_1$ equals the all-zero vector.
Any two distinct vectors from $\left\{v_1, \cdots, v_N\right\}$ are orthogonal, i.e., $v_{l}^{T}v_m = \delta_{lm}$ for $1 \leq i,m \leq N$.
Next, we consider three cases depending on the relationship between $i$ and $j$ for each matrix $Q_i$.\\
\textbf{Case 1}: $2 \leq j <i$. 
  Each component $\left(Q_iv_j \right)_m$ of the vector $Q_iv_j$ equals
\begin{equation*} 
  \begin{aligned}
    \left(Q_iv_j \right)_m&= \sum_{k=1}^N(Q_i)_{mk}(v_j)_k= \sum_{k=1}^j(Q_i)_{mk}(v_j)_k\\
    &= \begin{cases}
        (Q_i)_{mm}(v_j)_m = (v_j)_m   & \text{$m \leq j$,}\\
        \sum_{k=1}^j(Q_i)_{ik}(v_j)_k=-\sum_{k=1}^N (v_j)_k = 0& \text{$m=i$,}\\
        0& \text{otherwise.}
      \end{cases}
  \end{aligned} 
 \end{equation*}
Thus, we deduce $Q_iv_j = v_j$ for each vector $v_j$ with $2 \leq j <i$. \\
\textbf{Case 2}. $2 \leq j=i$. Each component $\left(Q_iv_i \right)_m$ of the vector $Q_iv_i$ equals
\begin{equation*}
  \begin{aligned}
     \left(Q_iv_i \right)_m&= \sum_{k=1}^i(Q_i)_{mk}(v_i)_k \\
     &=
      \begin{cases}
        (Q_i)_{mm}(v_i)_m +  (Q_i)_{mi}(v_i)_i = i = i(v_i)_m   & \text{$m < i$,}\\
        \sum_{k=1}^{i-1}(Q_i)_{ik}(v_i)_k+(Q_i)_{ii}(v_i)_i= i(1-i)=i(v_i)_i& \text{$m=i$,}\\
        0& \text{$m > i$.}
      \end{cases}
  \end{aligned}
\end{equation*}
which indicates  $Q_i v_i = i v_i$ for the vector $v_i$.\\
\textbf{Case 3}: $j>i$. Each component $\left(Q_iv_j \right)_m$ of the vector $Q_iv_j$ can be written as $\left(Q_iv_j \right)_m= \sum_{k=1}^N(Q_i)_{mk}(v_j)_k= \sum_{k=1}^i(Q_i)_{mk}(v_j)_k=\sum_{k=1}^i(Q_i)_{mk}=\sum_{k=1}^N(Q_i)=0$ for $1 \leq m \leq N$. 
    Consequently, we have $Q_i v_j = 0v_j$ for each vector $v_j$ with $ j >i$. 
\end{proof}

Now we compute the eigenvalues and eigenvectors of $Q_W \in {\cal A}_N$.

\begin{theorem}\label{the:th}
    Let $G_{W}$ be a weighted threshold graph on $N$ nodes with weight vector 
    $W=(w_2, \ldots , w_N)$.
    The spectrum of the Laplacian $Q_W$ of $G_W$ is $\left\{ 0, \mu_2, \ldots, \mu_N \right\}$ where
    \begin{equation}\label{eq:th}
         \mu_i=iw_i + \sum_{j=i+1}^{N} w_j.
    \end{equation}
    This can be written in a matrix form as 
    \begin{equation}\label{eq:th'}
        \mu^T=(\mu_2,\ldots,\mu_N)^T = UW^T,
    \end{equation}
    where $U$ is an $(N-1) \times (N-1)$ upper triangular matrix defined by
    \begin{equation} \label{eq:th''}
        U_{ij}=
            \begin{cases}
                1 &  \text{$ 1 \leq i < j \leq N-1$, }\\
                i+1 & \text{$1 \leq i=j \leq N-1$, }\\
                0 & \text{$otherwise$}.
            \end{cases}
    \end{equation}        
\end{theorem}

\begin{proof}
Let $v_j$ be a common eigenvector for 
$Q_2,\ldots,Q_N$ defined in Lemma \ref{lm:1}. 
Then 
\[
Q_W v_j =  \left( \sum_{i=2}^{N} w_iQ_i \right) v_j = j w_j v_j + \sum_{i=j+1} ^{N}w_i v_j =\left( jw_j + \sum_{k=j+1}^{N} w_k \right) v_j,
\]
for $j=2,\ldots,N$. 
This proves Eq.~\ref{eq:th}, and Eq.~\ref{eq:th'} follows straightforwardly.
\end{proof}

For the graph given in Figure~\ref{fig} with weight vector $W=(1,0,-\sqrt{2},0,~2)$ Eq.~\ref{eq:th'} gives
\begin{equation}
\mu^T = UW^T =
    \left[
        \begin{array}{rrrrr}
        2 & 1 & 1 & 1 & 1  \\
        0 & 3 & 1 & 1 & 1  \\
        0 & 0 & 4 & 1 & 1  \\
        0 & 0 & 0 & 5 & 1 \\
        0 & 0 & 0 & 0 & 6
        \end{array}\right]
        \left[
        \begin{array}{c}
        1   \\
        0   \\
        -\sqrt{2}   \\
        0  \\
        2 
        \end{array}\right]
        =
        \left[
        \begin{array}{c}
        4-\sqrt{2}   \\
        2-\sqrt{2}  \\
        2-4\sqrt{2}   \\
        2  \\
        12 
        \end{array}\right].
\end{equation}
Thus the spectrum of $Q_W$ is equal to $\{0,\ 4-\sqrt{2},\ 2-\sqrt{2},\ 2-4\sqrt{2},\ 2,\   12\}$.
\\

Independently Anđelić and Stanic \cite{andelic2025laplacian} also found a closed formula of the Laplacian eigenvalues of weighted threshold graphs. 
Their proof uses weighted Ferrers diagrams and is different from ours.

Theorem~\ref{the:th} indicates that weighted threshold graphs with integral weights have integral spectrum.
Since $U$ is an upper triangular matrix with nonzero diagonal, $U$ is invertable and it is straightforward to compute its inverse $U^{-1}$,
\begin{equation}
U^{-1}=
\begin{bmatrix}
    \frac{1}{2} & -\frac{1}{6} & -\frac{1}{12}  & -\frac{1}{20} &\cdots & -\frac{1}{(N-1)N}   \\
      & \frac{1}{3} & -\frac{1}{12} & -\frac{1}{20}   & \cdots & -\frac{1}{(N-1)N}  \\
      &        & \frac{1}{4}  & -\frac{1}{20} & \cdots & -\frac{1}{(N-1)N} \\
      &        &    & \frac{1}{5}  & \cdots    &-\frac{1}{(N-1)N} \\
      &        &\text{\huge0}    &     & \ddots & \vdots     \\
      &        &    &     &        & \frac{1}{N}  \\     
\end{bmatrix},
\end{equation}
and $(w_2,\ldots,w_N)^T = U^{-1}(\mu_2,\ldots,\mu_N)^T$.
As a result, we have the following corollary.

\begin{corollary}\label{th:lse}
    For any given real vector $\mu=(0,\mu_2,\ldots,\mu_N)$ there exists a weighted threshold graph of order $N$, whose Laplacian spectrum is 
    $\{0,\mu_2,\ldots,\mu_N\}$.
\end{corollary}

A different ordering of $\mu_2,\ldots,\mu_N$ gives a different weight vector, and therefore, a different weighted threshold graph with the same Laplacian spectrum.

The degree expression in Eq.~(\ref{eq:wd}) yields the following corollary.

\begin{corollary}\label{th:dl}
    Let $G_W$ be a weighted threshold graph of order $N$ with a weight vector $W=(w_2,\ldots,w_N)$.
    The Laplacian spectrum of $G_W$ is given by $\left\{ \mu_{1}=0,\mu_2,\ldots,\mu_{N} \right\}$ where for $2 \leq i \leq N$
    \begin{equation*}
        \mu_i=d_i + w_i,
    \end{equation*}
where $d_i$ denotes the weighted degree for the node $i$.
\end{corollary}

\section{Cospectrality}
\label{sec:cos}
Two matrices are called cospectral if they have the same eigenvalues.
In the previous section, we observed that different weighted threshold graphs can have cospectral Laplacian matrices.
We will show that the observation is not the case if the number of distinct weights is limited to three.
We start with an observation about the Laplacian matrix of arbitrary weighted graphs (see \cite{van2023graph}, \textbf{Art.} 125).

\begin{proposition}\label{prp:ab}
Consider two weighted graphs of order $N$ with cospectral Laplacian matrices $Q$ and $Q'$.
Then $aQ+b(J_N-NI_N)$ and $aQ'+b(J_N-NI_N)$ are also cospectral Laplacian matrices for any real numbers $a$ and $b$, 
where $J_N$ is the $N\times N$ all-one matrix and $I_N$ is the identity matrix of order $N$.
\end{proposition}

\begin{proof}
Let $v_1,\ldots,v_N$ be an orthogonal basis of eigenvectors for Laplacian $Q$ corresponding to eigenvalues 
$\mu_1,\mu_2,\ldots,\mu_N$, such that $\mu_1=0$ and $v_1=\1$ is the all-ones vector.
Then $v_1,\ldots,v_N$ is also a set of eigenvectors for $J_N$ with eigenvalues $N,0,\ldots,0$.
Therefore $aQ+b(J_N-NI_N)$ has eigenvalues $0,a\mu_2-bN,\ldots,a\mu_N-bN$.
Since $Q'$ has the same spectrum as $Q$, we find by the same argument that 
$aQ'+b(J_N-NI_N)$ also has eigenvalues $0,a\mu_2-bN,\ldots,a\mu_N-bN$.
\end{proof}

\begin{theorem}\label{the:3w}
    Consider two threshold graphs $G_W$ and $G_{W'}$ with $N$ nodes and weight vectors $W$ and $W'$.
    Suppose that all weights are taken from $\{x_1,x_2,x_3\}$.
    If $G_W$ and $G_{W'}$ have cospectral Laplacian matrices, then $G_W$ and $G_{W'}$ are isomorphic.
\end{theorem}

\begin{proof} 
For convenience. we first apply Proposition~\ref{prp:ab} and assume without loss of generality that
$-1=x_1\leq x_2\leq x_3=1$.
Suppose $G_W$ and $G_{W'}$ both have spectrum $\{0,\mu_2,\ldots,\mu_N\}$,
and define $\mu_{max}=\max\{\mu_2,\ldots,\mu_N\}$ and $\mu_{min}=\min\{\mu_2,\ldots,\mu_N\}$.
We apply Eq.~(\ref{eq:th''}).
The matrix $U$ has all row sums equal to $N$, no negative entries, and only positive entries in the last column. 
This implies that $U\1=N\1$, and $Uv<N\1$ for every vector $v<\1$ (recall that $\1$ is the all-one vector). 
Suppose $w_N=1$.
Then clearly $W^T\leq \1$, $\mu^T=UW^T\leq N\1$, and $\mu_N=N$. 
This implies that $\mu_{max}=N$.
Next, suppose $w_N<1$, then $W^T<\1$ and, therefore, $UW^T<N\1$, hence $\mu_{max}<N$.
Thus, we conclude that $w_N=1$ if and only if $\mu_{max}=N$.
Similarly, $w_N=-1$ if and only if $\mu_{min}=-N$. Therefore, 
$w_N=x_2$ if and only if $-N<\mu_{min}$ and $\mu_{max}<N$.
The same holds for $G_{W'}$ and, hence, $w_N=w'_N$.
Now, the result follows by deleting node $N$ and applying induction.
\end{proof}

If the weights take more than three distinct values, then two weighted threshold graphs with cospectral Laplacian matrices need not be isomorphic as is shown by the following two matrices (both have spectrum $\{0,0,6\}$).
\begin{equation*}
    \begin{bmatrix}
    3 & -3 & 0 \\
    -3 & 3 & 0 \\
    0 & 0 & 0\\
    \end{bmatrix},
    \begin{bmatrix}
    1 & 1 & -2\\
    1 & 1 & -2 \\
    -2 & -2 & 4\\
    \end{bmatrix}.
\end{equation*}

If $x_1=-1$, $x_2=0$, and $x_3=1$, then the considered weighted threshold graphs are signed graphs. 
Signed threshold graphs are studied in \cite{andjelic2023signed}, 
where the name {\lq}net-Laplacian{\rq} instead of {\lq}Laplacian{\rq} is used.
(because the Laplacian matrix of a signed graph has a different meaning). 
Two signed threshold graphs with cospectral net-Laplacian matrices must be isomorphic \cite{andjelic2023signed}.
Our theorem generalizes the result in \cite{andjelic2023signed} to three possible weights with arbitrary real value.

Theorem~\ref{the:3w} does not imply that in case of three weights, a weighted threshold graph is determined by its Laplacian spectrum.
There could be a weighted graph with the same weights and the same Laplacian spectrum, which is not a threshold graph. 
However, we do not expect that this will happen.
It is known \cite{hammer1996laplacian} that unweighted threshold graphs are determined by their Laplacian spectrum.
This means that any unweighted graph with the same Laplacian spectrum as 
a threshold graph $G$ must be isomorphic to $G$.
Proposition~\ref{prp:ab} shows that in the case of only two possible weights, a weighted threshold graph is also determined by its Laplacian spectrum.
In particular, if all weights are $\pm 1$ the matrix is known as the Seidel Laplacian,
and thus we conclude that threshold graphs are determined by the spectrum
of the Seidel Laplacian matrix.
\\

\noindent
{\bf Acknowledgement.}
We thank a referee of an earlier version of this paper for pointing at reference~\cite{andelic2025laplacian}.

\bibliographystyle{plain}
\bibliography{tg}

\end{document}